\documentclass[12pt,reqno]{amsart}
\usepackage{amssymb}

\large\normalsize
\theoremstyle{plain}
\newtheorem{theorem}{Theorem}
\newtheorem{corollary}{Corollary}

\theoremstyle{definition}

\newtheorem{remark}{Remark}

\numberwithin{equation}{section}
\numberwithin{theorem}{section}
\numberwithin{corollary}{section}
\numberwithin{definition}{section}
\numberwithin{remark}{section}
\numberwithin{lemma}{section}

\begin{document}

\title[Positive Green's functions for fractional bvps] {Positive Green's functions for some fractional-order boundary value problems}
\author{Douglas R. Anderson} 
\address{Department of Mathematics \\
         Concordia College \\
         Moorhead, MN 56562 USA}
\email{andersod@cord.edu}


\begin{abstract}
We use the newly introduced conformable fractional derivative, which is different from the Caputo and Riemann-Liouville fractional derivatives, to reformulate several common boundary value problems, including those with conjugate, right-focal, and Lidstone conditions. With the fractional differential equation and fractional boundary conditions established, we find the corresponding Green's functions and prove their positivity under appropriate assumptions. 
\end{abstract}

\maketitle\thispagestyle{empty}

\section{Introduction}

The search for the existence of positive solutions and multiple positive solutions to nonlinear fractional boundary value problems has expanded greatly over the past decade; for some recent examples please see [1-7,9,12-16]. In all of these works and the references cited therein, however, the definition of the fractional derivative used is either the Caputo or the Riemann-Liouville fractional derivative, involving an integral expression and the gamma function. Recently \cite{udita,hammad,khalil} a new definition has been formulated and dubbed the conformable fractional derivative. In this paper, we use this fractional derivative of order $\alpha$, given by
\begin{equation}\label{derivdef}
 D^{\alpha}f(t):=\lim_{\varepsilon\rightarrow 0}\frac{f(te^{\varepsilon t^{-\alpha}})-f(t)}{\varepsilon}, \quad D^{\alpha}f(0)=\lim_{t\rightarrow 0^+}D^{\alpha}f(t);  
\end{equation}
note that if $f$ is differentiable, then
\begin{equation}\label{fracshort}
 D^{\alpha}f(t) = t^{1-\alpha} f'(t), 
\end{equation}
where $f'(t)=\lim_{\varepsilon\rightarrow 0}[f(t+\varepsilon)-f(t)]/\varepsilon$.
Using this new definition of the fractional derivative, we reformulate several common boundary value problems, including those with conjugate, right-focal, and Lidstone conditions. With the fractional differential equation and fractional boundary conditions established, we find the corresponding Green's functions and prove their positivity under appropriate assumptions. This work thus sets the stage for using fixed point theorems to prove the existence of positive and multiple positive solutions to nonlinear fractional problems based on the conformable fractional derivative and these boundary value problems, as the kernel of the integral operator is often Green's function.

\section{Two Iterated Fractional Derivatives}

We begin by considering two iterated fractional derivatives in the differential operator, together with two-point boundary conditions, as illustrated in the nonlinear boundary value problem 
	\begin{gather}
	-D^{\beta}D^{\alpha} x(t) = f(t,x(t)), \qquad 0 \le t \le 1, \label{2eigeneqn} \\
	\gamma x(0)-\delta D^\alpha x(0) = 0 = \eta x(1)+\zeta D^{\alpha} x(1), \label{2eigenbc}
	\end{gather}
where	$\alpha,\beta\in(0,1]$ and the derivatives are conformable fractional derivatives \eqref{derivdef}, with $\gamma,\delta,\eta,\zeta\ge 0$ and  $d:=\eta\delta+\gamma\zeta+\gamma\eta/\alpha>0$. To verify the existence of positive solutions to this problem, one often finds the corresponding Green's function, shows that it is non-negative, and then uses a fixed point theorem on an integral operator whose kernel is Green's function. Throughout this and subsequent sections, we will state the fractional boundary value problem, find Green's function for it, and prove that it is positive except at various boundary points.


\begin{theorem}
Let $\alpha,\beta\in(0,1]$. The corresponding Green's function for the homogeneous problem 
\[ -D^{\beta}D^{\alpha} x(t)=0 \] 
satisfying boundary conditions \eqref{2eigenbc} is given by 
\begin{equation}\label{gf2}
	G(t,s)=
			\begin{cases}
				\frac{1}{d}\left[\delta+\frac{\gamma}{\alpha} t^\alpha\right]\left[\zeta+\frac{\eta}{\alpha}\left(1-s^\alpha\right)\right] &:t\le s \\
				\frac{1}{d}\left[\delta+\frac{\gamma}{\alpha} s^\alpha\right]\left[\zeta+\frac{\eta}{\alpha}\left(1-t^\alpha\right)\right] &:s\le t
			\end{cases}
\end{equation}
where we assume with $\gamma,\delta,\eta,\zeta\ge 0$ and $d=\eta\delta+\gamma\zeta+\gamma\eta/\alpha>0$.
\end{theorem}

\begin{proof}
Let $h$ be any continuous function. We will show that
\[ x(t)=\int_{0}^{1}G(t,s)h(s)s^{\beta-1}ds, \]
for $G$ given by \eqref{gf2}, is a solution to the linear boundary value problem
\[ -D^{\beta}D^{\alpha} x(t) = h(t) \]
with boundary conditions \eqref{2eigenbc}. 

For any $t\in[0,1]$, using the branches of \eqref{gf2} we have
\begin{eqnarray*} 
 x(t) &=& \frac{1}{d}\left[\zeta+\frac{\eta}{\alpha}\left(1-t^\alpha\right)\right] \int_{0}^{t} \left[\delta+\frac{\gamma}{\alpha} s^\alpha\right] h(s)s^{\beta-1}ds \\
  & & + \frac{1}{d}\left[\delta+\frac{\gamma}{\alpha} t^\alpha\right] \int_{t}^{1} \left[\zeta+\frac{\eta}{\alpha}\left(1-s^\alpha\right)\right]h(s)s^{\beta-1}ds.
\end{eqnarray*}
Taking the $\alpha$-fractional derivative yields
\begin{eqnarray*} 
 D^{\alpha}x(t) &=& -\frac{\eta}{d} \int_{0}^{t} \left[\delta+\frac{\gamma}{\alpha} s^\alpha\right] h(s)s^{\beta-1}ds 
   + \frac{\gamma}{d} \int_{t}^{1} \left[\zeta+\frac{\eta}{\alpha}\left(1-s^\alpha\right)\right]h(s)s^{\beta-1}ds.
\end{eqnarray*}
Checking the first boundary condition, we see that
\[ \gamma x(0)-\delta D^\alpha x(0) = 0. \]
Moreover, in checking the second boundary condition we get
\[ \eta x(1)+\zeta D^{\alpha} x(1) = 0. \]
Taking the $\beta$-fractional derivative of the $\alpha$-fractional derivative yields
\begin{eqnarray*} 
 D^{\beta}D^{\alpha}x(t) &=& -\frac{\eta}{d} \left[\delta+\frac{\gamma}{\alpha} t^\alpha\right] h(t)t^{\beta-1}t^{1-\beta} 
   - \frac{\gamma}{d} \left[\zeta+\frac{\eta}{\alpha}\left(1-t^\alpha\right)\right]h(t)t^{\beta-1}t^{1-\beta} \\
   &=& -\frac{1}{d}h(t)\left[ \eta\delta  + \gamma\zeta+\frac{\gamma\eta}{\alpha}\right] = -h(t),
\end{eqnarray*}
which is what we set out to prove.
\end{proof}


\begin{corollary}[Fractional Conjugate and Right-Focal Problems]
The corresponding Green's function for the homogeneous problem 
\[ -D^{\beta}D^{\alpha} x(t)=0 \] 
satisfying the conjugate boundary conditions $x(0)=x(1)=0$ is given by 
\begin{equation}\label{gfconj}
	G(t,s)=
			\begin{cases}
				\frac{1}{\alpha} t^\alpha\left(1-s^\alpha\right) &:t\le s, \\
				\frac{1}{\alpha} s^\alpha\left(1-t^\alpha\right) &:s\le t,
			\end{cases}
\end{equation}
and the corresponding Green's function for the homogeneous problem 
\[ -D^{\beta}D^{\alpha} x(t)=0 \] 
satisfying the right-focal-type boundary conditions $x(0)=D^{\alpha}x(1)=0$ is given by 
\begin{equation}\label{gfrtfoc}
	G(t,s)=
			\begin{cases}
				\frac{1}{\alpha} t^\alpha &:t\le s, \\
				\frac{1}{\alpha} s^\alpha &:s\le t.
			\end{cases}
\end{equation}
\end{corollary}


\begin{remark}
Note that the fractional Green's function given above for the conjugate boundary conditions in \eqref{gfconj} differs from that found for example in Bai and L\"{u} \cite{bai}.
\end{remark}


\begin{theorem}
For $G(t,s)$ given in \eqref{gf2}, we have that
\begin{equation}\label{Gbds2} 
  g(t)G(s,s) < G(t,s) \le G(s,s)
\end{equation}
for $t,s\in[0,1]$, where
\begin{equation}\label{g2}
  g(t):=\min\left\{\frac{\alpha\delta+\gamma  t^\alpha}{\alpha\delta+\gamma},\frac{\alpha\zeta+\eta\left(1-t^{\alpha}\right)}{\alpha\zeta+\eta}\right\}.
\end{equation}
\end{theorem}

\begin{proof}
It is straightforward to see that
$$ \frac{G(t,s)}{G(s,s)}=\begin{cases}
     \displaystyle\frac{\delta + \frac{\gamma}{\alpha}t^\alpha}{\delta + \frac{\gamma}{\alpha}s^\alpha}, &  t\le s, \\
       & \\
     \displaystyle\frac{\zeta+\frac{\eta}{\alpha}(1-t^\alpha)}{\zeta+\frac{\eta}{\alpha}(1-s^\alpha)}, & s\le t;
  \end{cases} $$
this expression yields both inequalities in \eqref{Gbds2} for $g$ as in \eqref{g2}.
\end{proof}


\begin{remark}
It is also common to skip the Green's function representation and to express solutions directly. For instance, let $\alpha,\beta\in(0,1]$. The motivated reader can verify that the fractional boundary value problem 
\[ -D^{\beta}D^{\alpha} x(t)=h(t), \quad t\in(0,1), \] 
satisfying the three-point boundary conditions 
\[ x(0)=0, \quad \delta x(\eta)=x(1), \]
where $h$ is a continuous function, $\eta\in(0,1)$ and $0<\delta \eta^\alpha<1$, is given by 
\begin{eqnarray*} 
  x(t) &=& \frac{-1}{\alpha}\int_0^t \left(t^\alpha-s^\alpha\right)h(s)s^{\beta-1}ds  - \frac{\delta  t^\alpha}{\alpha\left(1-\delta  \eta^\alpha\right)}\int_0^\eta \left(\eta^\alpha-s^\alpha\right)h(s)s^{\beta-1}ds \\
       & & + \frac{t^\alpha}{\alpha\left(1-\delta  \eta^\alpha\right)}\int_0^1 \left(1-s^\alpha\right)h(s)s^{\beta-1}ds. 
\end{eqnarray*}
\end{remark}

\section{Three Iterated Fractional Derivatives}

Next we consider the three iterated fractional derivative nonlinear right-focal problem
	\begin{gather}
	D^{\gamma}D^{\beta}D^{\alpha} x(t) = f(t,x(t)), \qquad 0 \le t \le 1, \label{eigeneqn} \\
	x(0) = D^{\alpha} x(\tau) = D^{\beta}D^{\alpha} x(1)=0, \label{eigenbc}
	\end{gather}
where	$\alpha,\beta,\gamma\in(0,1]$ with $0 < \tau < 1$. One could impose further the conditions $\alpha+\beta\in(1,2]$ and $\alpha+\beta+\gamma\in(2,3]$ if one wishes to explore a fractional problem of order $(2,3]$. Our approach to the existence of positive solutions would again involve Green's function for this problem. Once the following three theorems are established, an interested reader could then apply a fixed point theorem to get positive solutions to \eqref{eigeneqn}, \eqref{eigenbc}, although the details are omitted here.


\begin{theorem}[Fractional Right-Focal Problem]
Let $\alpha,\beta,\gamma\in(0,1]$ and $0 < \tau < 1$. The corresponding Green's function for the homogeneous problem 
\[ D^{\gamma}D^{\beta}D^{\alpha} x(t)=0 \] 
satisfying boundary conditions \eqref{eigenbc} is given by 
\begin{equation}\label{greensfunction}
	G(t,s)=
	\begin{cases}
    	s\in [0,\tau] &:
			\begin{cases}
				u(t,s)	 &:t\le s \\
        x(0,s) &:s\le t
			\end{cases}\\
			&  \\
      s\in [\tau,1] &:
			\begin{cases}
				u(t,\tau)	      &:t\le s \\
				u(t,\tau)+x(t,s) &:s\le t
			\end{cases} \\
	\end{cases}
\end{equation}
where
\begin{equation}\label{udef}
  u(t,s) = \frac{(\alpha+\beta)t^{\alpha}s^{\beta}-\alpha  t^{\alpha+\beta}}{\alpha\beta(\alpha+\beta)} 
\end{equation}
and $x(\cdot,\cdot)$ is the Cauchy function given
\begin{equation}\label{xdef} 
  x(t,s) = \frac{\alpha t^{\alpha}\left(t^{\beta}-s^{\beta}\right)+\beta s^{\beta}\left(s^{\alpha}-t^{\alpha}\right)}{\alpha\beta(\alpha+\beta)}.
\end{equation}
\end{theorem}

\begin{proof}
Let $h$ be any continuous function. We will show that
\[ x(t)=\int_{0}^{1}G(t,s)h(s)s^{\gamma-1}ds, \]
for $G$ given by \eqref{greensfunction}, is a solution to the linear boundary value problem
\[ D^{\gamma}D^{\beta}D^{\alpha} x(t)=h(t) \]
with boundary conditions \eqref{eigenbc}. 

First let $t\in[0,\tau]$. Then
\begin{equation}\label{xcase1} 
  x(t) = \int_{0}^{t} x(0,s)h(s)s^{\gamma-1}ds + \int_{t}^{\tau} u(t,s)h(s)s^{\gamma-1}ds + \int_{\tau}^{1} u(t,\tau)h(s)s^{\gamma-1}ds;
\end{equation}
clearly $x(0)=0$ by \eqref{udef}. Differentiating \eqref{xcase1}, we have
\begin{eqnarray*}
 D^{\alpha}x(t) &=& x(0,t)h(t)t^{\gamma-1}t^{1-\alpha} + \int_{t}^{\tau} \left(\frac{s^\beta-t^\beta}{\beta}\right)h(s)s^{\gamma-1}ds \\
 & & - u(t,t)h(t)t^{\gamma-1}t^{1-\alpha} + \int_{\tau}^{1} \left(\frac{\tau^\beta-t^\beta}{\beta}\right)h(s)s^{\gamma-1}ds \\
 &=& \frac{1}{\beta}\int_{t}^{\tau} \left(s^\beta-t^\beta\right)h(s)s^{\gamma-1}ds + \frac{1}{\beta}\int_{\tau}^{1} \left(\tau^\beta-t^\beta\right)h(s)s^{\gamma-1}ds.
\end{eqnarray*}
Clearly the second boundary condition $D^{\alpha}x(\tau)=0$ is met. Differentiating again, we have
\begin{eqnarray*}
 D^{\beta}D^{\alpha}x(t) &=& \frac{1}{\beta}\int_{t}^{\tau} \left(-\beta t^{\beta-1}t^{1-\beta}\right)h(s)s^{\gamma-1}ds + \frac{1}{\beta}\int_{\tau}^{1} \left(-\beta t^{\beta-1}t^{1-\beta}\right)h(s)s^{\gamma-1}ds \\
 &=& \int_{1}^{t} h(s)s^{\gamma-1}ds.
\end{eqnarray*}
It follows that $D^{\beta}D^{\alpha} x(1)=0$ and $D^{\gamma}D^{\beta}D^{\alpha} x(t)=h(t)$, proving the claim in this case.

Next let $t\in[\tau,1]$. Then
\begin{equation}\label{xcase2} 
  x(t) = \int_{0}^{\tau} x(0,s)h(s)s^{\gamma-1}ds + \int_{\tau}^{t} x(t,s)h(s)s^{\gamma-1}ds + u(t,\tau)\int_{\tau}^{1} h(s)s^{\gamma-1}ds;
\end{equation}
again $x(0)=0$ by \eqref{udef}. Differentiating \eqref{xcase2}, we have
\begin{eqnarray*}
 D^{\alpha}x(t) &=& \frac{1}{\beta}\int_{\tau}^{t} \left(t^\beta-s^\beta\right)h(s)s^{\gamma-1}ds + \left(\frac{\tau^\beta-t^\beta}{\beta}\right)\int_{\tau}^{1} h(s)s^{\gamma-1}ds.
\end{eqnarray*}
The second boundary condition $D^{\alpha}x(\tau)=0$ is clearly met. Differentiating again, we have
\begin{eqnarray*}
 D^{\beta}D^{\alpha}x(t) &=& \int_{1}^{t} h(s)s^{\gamma-1}ds;
\end{eqnarray*}
as in the previous case, $D^{\beta}D^{\alpha} x(1)=0$ and $D^{\gamma}D^{\beta}D^{\alpha} x(t)=h(t)$, proving the claim in this case as well.
\end{proof}


\begin{theorem}
For $G(t,s)$ given in \eqref{greensfunction}, we have that
\begin{equation}\label{Gbds} 
  0 < G(t,s) \le G(\tau,s)
\end{equation}
for $t\in(0,1]$ and $s\in(0,1]$, provided the condition $u(1,\tau)>0$ is met for $u$ in \eqref{udef}, that is to say the inequality
\begin{equation}\label{boundarydistance}
  \tau > \left(\frac{\alpha}{\alpha+\beta}\right)^{1/\beta}
\end{equation}
holds. 
\end{theorem}

\begin{proof}
Referring to \eqref{greensfunction}, \eqref{udef}, and \eqref{xdef}, we see that 
\[ u(s,s)=x(0,s), \quad x(0,\tau)=u(t,\tau)+x(t,\tau), \]
ensuring that $G(t,s)$ is a well-defined function; we also see that $G(t,s)=0$ if $t=0$ or $s=0$. From \eqref{udef} specifically, $u(0,s)=0$  for all $s$ and
\[ \frac{d}{dt}u(t,s)=\frac{t^{\alpha-1}(s^\beta-t^\beta)}{\beta} \ge 0, \quad s\ge t. \]
Moreover, from \eqref{xdef} we have $x(s,s)=0$ and
\[ \frac{d}{dt}x(t,s)=\frac{t^{\alpha-1}(t^\beta-s^\beta)}{\beta} \ge 0, \quad t\ge s, \]
so that $G(t,s)$ is non-decreasing on $[0,\tau]$ and non-increasing on $[\tau,1]$. Thus \eqref{Gbds} will hold if $G(1,\tau)>0$ holds, which occurs if $u(1,\tau)>0$. This is equivalent to \eqref{boundarydistance}, completing the proof.
\end{proof}


\begin{theorem}
For all $t,s\in[0,1]$,
\begin{equation}\label{greenbounds}
 g(t)G(\tau,s)\le G(t,s)\le G(\tau,s)
\end{equation}
where
\begin{equation}\label{g}
   g(t):=\min\left\{\frac{t^\alpha}{\beta \tau^{\alpha+\beta}}\left[(\alpha+\beta)\tau^\beta-\alpha t^\beta\right],\frac{1-t}{1-\tau}\right\}.
\end{equation}
\end{theorem}

\begin{proof}
From the preceeding theorem, we have $G(t,s)\le G(\tau,s)$ for all $t,s\in[0,1]$. For the lower bound, we proceed by cases on the branches of the Green's function \eqref{greensfunction}, that is we use \eqref{udef} and \eqref{xdef}.
\begin{enumerate}
  \item[$(i)$] $0 \le t \le s \le \tau$: Here $G(t,s)=u(t,s)$, $G(\tau,s)=x(0,s)=\frac{1}{\alpha(\alpha+\beta)}s^{\alpha+\beta}$.  For these $t,s$ we have
      \[ \frac{u(t,s)}{x(0,s)} \ge \frac{u(t,\tau)}{x(0,\tau)} \ge \frac{t^\alpha}{\beta \tau^{\alpha+\beta}}\left[(\alpha+\beta)\tau^\beta-\alpha t^\beta\right]\]
      which implies
      \[  G(t,s) \ge \frac{t^\alpha}{\beta \tau^{\alpha+\beta}}\left[(\alpha+\beta)\tau^\beta-\alpha t^\beta\right] G(\tau,s). \]
  \item[$(ii)$] $0 \le t \le \tau \le s \le 1$: In this case $G(t,s)=u(t,\tau)$ and $G(\tau,s)=u(\tau,\tau)$, so again we have
      \[  G(t,s) \ge \frac{t^\alpha}{\beta \tau^{\alpha+\beta}}\left[(\alpha+\beta)\tau^\beta-\alpha t^\beta\right] G(\tau,s). \]
  \item[$(iii)$] $0 \le s \leq t \le \tau$ or  $0 \le s \le \tau \le t \le 1$: 
       Since $G(t,s)=G(\tau,s)=\frac{1}{\alpha(\alpha+\beta)}s^{\alpha+\beta}$, it follows that
      \[  G(t,s) = G(\tau,s). \]
  \item[$(iv)$] $\tau \le t \le s \le 1$: As in $(ii)$, $G(t,s)=u(t,\tau)$ and $G(\tau,s)=u(\tau,\tau)$.  Define
	\begin{align}
	w(t)&:=u(t,\tau)-\left(\frac{1-t}{1-\tau}\right)u(\tau,\tau)\label{w}\\
	&=G(t,s)-\left(\frac{1-t}{1-\tau}\right)G(\tau,s).\nonumber
	\end{align}
      Now $w(\tau)=0$, $w^{\prime}(\tau)>0$, and $w(1)=G(1,s)>0$ by \eqref{boundarydistance}.
      Since $w$ is concave down, $w(t)\geq 0$ on $[\tau,1]$, hence
      \[ G(t,s) \ge \left(\frac{1-t}{1-\tau}\right)G(\tau,s). \]
  \item[$(v)$] $\tau \le s \le t \le 1$: Note that $G(\tau,s)=u(\tau,\tau)$, while
      $G(t,s) = u(t,\tau)+x(t,s) \ge u(t,\tau)$;
      consequently, the employment of $w$ as in \eqref{w} yields
      \[ G(t,s) \ge \left(\frac{1-t}{1-\tau}\right)G(\tau,s). \]
\end{enumerate}
\end{proof}

\section{Four Iterated Fractional Derivatives}

In this final section we consider four iterated fractional derivatives in the differential operator, with two types of boundary conditions. First, consider the nonlinear two-point cantilever beam eigenvalue problem
	\begin{gather}
	D^{\delta}D^{\gamma}D^{\beta}D^{\alpha} x(t) = \lambda a(t)f(x), \qquad 0 \le t \le 1, \label{eigeneqn4} \\
	x(0) = D^{\alpha} x(0) = D^{\beta}D^{\alpha} x(1) = D^{\gamma}D^{\beta}D^{\alpha} x(1) = 0, \label{eigenbc4}
	\end{gather}
where	$\alpha,\beta,\gamma,\delta\in(0,1]$, and with $\alpha+\beta\in(1,2]$, $\alpha+\beta+\gamma\in(2,3]$, and $\alpha+\beta+\gamma+\delta\in(3,4]$ if one wishes to explore this problem as a fractional order between 3 and 4. The theme throughout this work has been to approach the existence of positive solutions to \eqref{eigeneqn4}, \eqref{eigenbc4} by involving Green's function for this problem. To obtain symmetry in Green's function below we must take $\gamma=\alpha$ in the fractional differential equation, but we will maintain the more general form ($\gamma$ not necessarily equal to $\alpha$) in the proofs to follow.


\begin{theorem}[Fractional Cantilever Beam]
Let $\alpha,\beta,\gamma,\delta\in(0,1]$. The corresponding Green's function for the homogeneous problem 
\begin{equation}\label{4thfrac}
 D^{\delta}D^{\gamma}D^{\beta}D^{\alpha} x(t)=0 
\end{equation} 
satisfying boundary conditions \eqref{eigenbc4} is given by 
\begin{equation}\label{gf4}
	G(t,s)=
	\begin{cases}
		\displaystyle\frac{t^{\alpha+\beta}}{\gamma}\left[\frac{s^\gamma}{\beta(\alpha+\beta)}-\frac{t^\gamma}{(\beta+\gamma)(\alpha+\beta+\gamma)}\right] &:t\le s, \\
    \displaystyle\frac{s^{\beta+\gamma}}{\alpha}\left[\frac{t^\alpha}{\beta(\beta+\gamma)}-\frac{s^\alpha}{(\alpha+\beta)(\alpha+\beta+\gamma)}\right] &:s\le t.
	\end{cases}
\end{equation}
\end{theorem}

\begin{proof}
Let $h$ be any continuous function. We will show that
\[ x(t)=\int_{0}^{1}G(t,s)h(s)s^{\delta-1}ds, \]
for $G$ given by \eqref{gf4}, is a solution to the linear boundary value problem
\[ D^{\delta}D^{\gamma}D^{\beta}D^{\alpha} x(t) = h(t) \]
with boundary conditions \eqref{eigenbc4}. 

For any $t\in[0,1]$, using the branches of \eqref{gf4} we have
\begin{eqnarray*} 
 x(t) &=& \int_{0}^{t} \frac{s^{\beta+\gamma}}{\alpha}\left[\frac{t^\alpha}{\beta(\beta+\gamma)}-\frac{s^\alpha}{(\alpha+\beta)(\alpha+\beta+\gamma)}\right] h(s)s^{\delta-1}ds \\
  & & - \frac{t^{\alpha+\beta}}{\gamma}\int_{1}^{t} \left[\frac{s^\gamma}{\beta(\alpha+\beta)}-\frac{t^\gamma}{(\beta+\gamma)(\alpha+\beta+\gamma)}\right]h(s)s^{\delta-1}ds;
\end{eqnarray*}
clearly $x(0)=0$. Taking the $\alpha$-fractional derivative yields
\begin{eqnarray*} 
 D^{\alpha}x(t) &=& \frac{1}{\beta(\beta+\gamma)}\int_0^t s^{\beta+\gamma+\delta-1}h(s)ds-\frac{t^\beta}{\beta\gamma}\int_1^t s^{\gamma+\delta-1}h(s)ds \\
 & & +\frac{t^{\beta+\gamma}}{\gamma(\beta+\gamma)}\int_1^ts^{\delta-1}h(s)ds.
\end{eqnarray*}
It is easy to see that $D^{\alpha}x(0)=0$. Taking the $\beta$-fractional derivative of the $\alpha$-fractional derivative yields
\begin{eqnarray*} 
 D^{\beta}D^{\alpha}x(t) &=& \frac{1}{\gamma}\int_1^t\left(t^\gamma-s^\gamma\right)s^{\delta-1}h(s)ds,
\end{eqnarray*}
so that $D^{\beta}D^{\alpha}x(1)=0$. Next,
\[ D^{\gamma}D^{\beta}D^{\alpha}x(t) = \int_1^t s^{\delta-1}h(s)ds, \]
from which we have $D^{\gamma}D^{\beta}D^{\alpha}x(1)=0$, and
\[ D^{\delta}D^{\gamma}D^{\beta}D^{\alpha}x(t) = h(t). \]
This finishes the proof.
\end{proof}


\begin{theorem} 
For all $t,s\in[0,1]$, Green's function given by \eqref{gf4} satisfies 
\[ 0 \le G(t,s) \le G(1,s). \]
\end{theorem}
 
\begin{proof} Note that
\begin{equation}\label{deltag} 
  \frac{d}{dt}G(t,s)=\begin{cases} 
     \frac{t^{\alpha+\beta-1}}{\beta\gamma(\beta+\gamma)}\left[\gamma  s^\gamma+\beta\left(s^\gamma-t^\gamma\right)\right] &: t \le s, \\ 
     \frac{s^{\beta+\gamma}t^{\alpha-1}}{\beta(\beta+\gamma)} &: s \le t. 
\end{cases} 
\end{equation} 
Fix $s\in[0,1]$. By the first boundary condition $G(0,s)=0$, and $\frac{d}{dt}G(t,s)$ as given above implies that $G(\cdot,s)$ is monotone increasing on $(0,1]$. In particular, $G(1,s)\ge G(t,s)\ge 0$ for all $t\in[0,1]$. 
\end{proof}

Finally, we will end the present discussion by considering another common set of boundary conditions, namely the so-called Lidstone conditions. Unlike the argument used below, one could also approach this problem as the conjunction of two conjugate problems, whose Green's function is given in \eqref{gfconj}. 


\begin{theorem}[Fractional Lidstone]
Let $\alpha,\beta\in(0,1]$. The symmetric Green's function for the homogeneous problem 
\begin{equation}\label{4thLidfrac}
 D^{\beta}D^{\alpha}D^{\beta}D^{\alpha} x(t)=0 
\end{equation} 
satisfying the Lidstone-type boundary conditions 
\begin{equation}
 x(0)=0=D^{\beta}D^{\alpha}x(0), \quad x(1)=0=D^{\beta}D^{\alpha}x(1)
\end{equation}
is given by 
\begin{equation}\label{lidgf}
	G(t,s)=
	\begin{cases}	
     u(t,s) &:t\le s, \\
     u(s,t) &:s\le t,
	\end{cases}
\end{equation}
where
\begin{equation}\label{lidu}
 u(t,s) = \frac{t^{\alpha}}{\alpha\beta(\alpha+\beta)(2\alpha+\beta)}\left[2\alpha s^\alpha\left(1-s^\beta\right)-\beta\left(1-s^\alpha\right)\left(t^{\alpha+\beta}+s^{\alpha+\beta}\right)\right]. 
\end{equation}
\end{theorem}

\begin{proof}
In this proof we construct Green's function from scratch, modifying the classical approach, found for example in \cite[Chapter 5]{kp}.
Let $x(t,s)$ be the Cauchy function associated with \eqref{4thfrac}, namely a function satisfying
\[ x(s,s)=D^{\alpha}x(s,s)=D^{\beta}D^{\alpha}x(s,s)=0, \quad D^{\gamma}D^{\beta}D^{\alpha}x(t,s)=1. \]
(Note that we will use $\gamma$ for now, and take $\gamma=\alpha$ for symmetry purposes at the conclusion.) Using \eqref{fracshort} at each step, it is easy to verify that here the Cauchy function is given by
\begin{equation}\label{Cauchyf}
 x(t,s) = \frac{1}{\gamma}\int_s^t\int_s^\tau \left(\xi^\gamma-s^\gamma\right)\xi^{\beta-1}\tau^{\alpha-1}d\xi d\tau, 
\end{equation}
and Green's function takes the form of
\[ G(t,s) = \begin{cases}	
     u(t,s) &:t\le s, \\
     u(t,s)+x(t,s) &:s\le t,
	\end{cases} \]
where $u(t,s)$ satisfies \eqref{4thfrac} and the two boundary conditions set at $t=0$. Thus
\[ u(t,s) = a(s)+b(s)t^\alpha+c(s)t^{\alpha+\beta}+d(s)t^{\alpha+\beta+\gamma}; \]
the Lidstone boundary conditions force $a(s)=c(s)=0$. The two boundary conditions at $t=1$ are satisfied by $(u+x)$ for $x$ given in \eqref{Cauchyf}. In particular, we use 
\[ D^{\beta}D^{\alpha}[u(t,s)+x(t,s)]|_{t\rightarrow 1}=0 \]
to solve for $d$, leading to
\[ d(s)=\frac{s^\gamma-1}{\gamma(\beta+\gamma)(\alpha+\beta+\gamma)}, \]
and then $u(1,s)+x(1,s)=0$ to solve for $b$, which yields
\begin{eqnarray*}
  b(s) &=& \frac{s^{\alpha+\beta+\gamma}}{\alpha(\alpha+\beta)(\alpha+\beta+\gamma)} -\frac{s^{\beta+\gamma}}{\alpha\beta(\beta+\gamma)}\\
       & & \; + \frac{s^\gamma}{\gamma}\left(\frac{1}{\beta(\alpha+\beta)}-\frac{1}{(\beta+\gamma)(\alpha+\beta+\gamma)}\right). 
\end{eqnarray*}
Altogether we have
\begin{eqnarray*}
 u(t,s) &=& s^\gamma t^\alpha \left(\frac{s^{\alpha+\beta}}{\alpha(\alpha+\beta)(\alpha+\beta+\gamma)} +\frac{\alpha+2\beta+\gamma}{\beta(\alpha+\beta)(\beta+\gamma)(\alpha+\beta+\gamma)} \right. \\
 & & \left. -\frac{s^\beta}{\alpha\beta(\beta+\gamma)}\right) + \frac{(s^\gamma-1)t^{\alpha+\beta+\gamma}}{\gamma(\beta+\gamma)(\alpha+\beta+\gamma)}.
\end{eqnarray*}
To achieve symmetry and to satisfy \eqref{4thLidfrac}, we take $\gamma=\alpha$ and arrive at \eqref{lidu}.
\end{proof}


\begin{theorem} 
For all $t,s\in(0,1)$, Green's function given by \eqref{lidgf} satisfies 
\[ G(t,s) > 0. \]
\end{theorem}
 
\begin{proof} 
Clearly $u(1,s)=u(t,1)=0$, and 
\[ u(t,s) \ge \frac{2s^{\alpha}t^{\alpha}}{\alpha\beta(\alpha+\beta)(2\alpha+\beta)}\left[\alpha \left(1-s^\beta\right)-\beta\left(1-s^\alpha\right)s^{\beta}\right] \]
for $0\le t\le s$. Define
\[ k(s):=\alpha \left(1-s^\beta\right)-\beta\left(1-s^\alpha\right)s^{\beta} = \beta s^{\alpha+\beta}-(\alpha+\beta)s^\beta+\alpha. \]
Then $k(0)=\alpha>0$, $k(1)=0$, and 
\[ k'(s)=\beta(\alpha+\beta)s^{\beta-1}\left(s^{\alpha}-1\right)<0. \]
Thus, $k$ is strictly decreasing, so that $k(s)>0$ for $s\in[0,1)$, forcing $u(t,s)\ge 0$ for $t,s\in[0,1]$. Therefore the result holds.
\end{proof}


\end{document}